\UseRawInputEncoding
\documentclass[reqno,10pt]{amsart}
\usepackage{amscd,amssymb}
\usepackage{latexsym}
\usepackage[all]{xy}

\def\:{{\colon}}

\def\into{\hookrightarrow}

\def\toisom{\widetilde{\to}}

\def\.{,\dots ,}
\def\wt{\widetilde}
\def\wh{\widehat}
\def\ol{\overline}
\def\wtimes{\wh{\otimes}}

\def\Res{{\rm Res}}

\def\Coh{{\rm Coh}}

\def\Spf{{\rm Spf}}

\def\Rad{{\rm Rad}}

\def\Spa{{\rm Spa}}

\def\Spec{{\rm Spec}}
\def\Proj{{\rm Proj}}
\def\Frac{{\rm Frac}}

\def\Bl{{\rm Bl}}

\def\ad{{\rm ad}}
\def\an{{\rm an}}

\def\Ker{{\rm Ker}}

\def\Pic{{\rm Pic}}

\def\bfA{{\bf A}}

\def\bfP{{\bf P}}

\def\bfR{{\bf R}}

\def\bfZ{{\bf Z}}

\def\gtI{{\mathfrak I}}
\def\gtJ{{\mathfrak J}}

\def\gtX{{\mathfrak X}}
\def\gtY{{\mathfrak Y}}

\def\bbA{{\mathbb A}}

\def\bbC{{\mathbb C}}

\def\calA{{\mathcal A}}
\def\calB{{\mathcal B}}
\def\calC{{\mathcal C}}
\def\calD{{\mathcal D}}

\def\calF{{\mathcal F}}

\def\calH{{\mathcal H}}

\def\calM{{\mathcal M}}

\def\calO{{\mathcal O}}

\def\calW{{\mathcal W}}
\def\calX{{\mathcal X}}

\def\oU{{\ol U}}

\def\oW{{\ol W}}
\def\oX{{\ol X}}

\def\oZ{{\ol Z}}

\def\of{{\ol f}}

\def\tilX{{\wt X}}

\def\tilx{{\wt x}}

\def\hatA{{\wh A}}

\def\hatX{{\wh X}}
\def\hatY{{\wh Y}}

\def\hatcalF{{\wh\calF}}

\def\kcirc{k^\circ}

\def\Kcirc{K^\circ}

\def\alp{{\alpha}}

\def\veps{\varepsilon}

\def\wHx{{\widetilde{\calH(x)}}}

\def\R+*{{\bf R^*_+}}

\def\ur{{\underline{r}}}

\def\us{{\underline{s}}}
\def\ut{{\underline{t}}}
\def\uf{{\underline{f}}}
\def\uT{{\underline{T}}}
\def\ua{{\underline{a}}}

\newtheorem{theor}[subsubsection]{Theorem}
\newtheorem{prop}[subsubsection]{Proposition}
\newtheorem{lem}[subsubsection]{Lemma}

\theoremstyle{definition}

\newtheorem{defin}[subsubsection]{Definition}
\newtheorem{rem}[subsubsection]{Remark}
\newtheorem{exam}[subsubsection]{Example}

\begin{document}
\title{Berkovich spaces and tubular descent}
\author{Oren Ben-Bassat, Michael Temkin}
\thanks{
The authors would like to thank the anonymous referee for the detailed comments and corrections in response to the initial submission. The first author would like to thank the University of Pennsylvania and the University of Haifa for travel support, as well as Jonathan Block, Ron Donagi, David Harbater, and Tony Pantev for helpful discussions. The second author would like to thank V. Drinfeld for the invitation to visit the University of Chicago and his support, and for informing him about the correspondence with V. Berkovich. The work of M.T. was partially supported by ISF grant No. 1018/11.}
%\address{\tiny{}}
\email{\scriptsize{ben-bassat@math.haifa.ac.il, temkin@math.huji.ac.il}}

\begin{abstract} We consider an algebraic variety $X$ together with the choice of a subvariety $Z$.  We show that any coherent sheaf on $X$ can be constructed out of a coherent sheaf on the formal neighborhood of $Z$, a coherent sheaf on the complement of $Z$, and an isomorphism between certain representative images of these two sheaves in the category of coherent sheaves on a  Berkovich analytic space $W$ which we define.
\end{abstract}


\maketitle

\section{Introduction}

The aim of this paper is to study how a coherent sheaf $\calF$ on a $k$-scheme of finite type $X$ can be reconstructed from its restrictions to an open subvariety $U$ and the complement $Z=X-U$, or informally speaking, how $\calF$ can be glued from $\calF_U$ and $\calF_Z$. Clearly, $\calF_U$ and $\calF_Z$ do not determine $\calF$, so one should suggest a richer descent datum which does determine $\calF$ but uses as little as possible information beyond the pair $(\calF_U,\calF_Z)$.

Note that at the very least one should know the restrictions of $\calF$ onto all closed subschemes whose reduction is $Z$, or equivalently one should know the formal completion $\hatcalF_Z$ of $\calF$ along $Z$. Furthermore, even the knowledge of $\hatcalF_Z$ and $\calF_U$ is insufficient, so we will introduce an intermediate geometric space $W$ that will play the role of the ``intersection" of $\gtX=\hatX_Z$ and $U$, and we will construct ``restriction" functors $\Coh(U)\to\Coh(W)$ and $\Coh(\gtX)\to\Coh(W)$. Then our main result can be formulated in the very intuitive way that $$\Coh(X)\toisom\Coh(U)\times_{\Coh(W)}\Coh(\gtX)$$ For reasons explained below we call this type of descent {\em tubular descent}. We remark that the essentially new case is when $Z$, in some sense, has a non-trivial global geometry with respect to $X$, e.g. $Z$ is a curve with a positive self-intersection in $\bfP^2$ (see \S\ref{relatedsec}).

\subsection{Choice of $W$}
Before explaining what $W$ is, let us consider the case that $k=\bbC$. Then one often considers a tubular neighborhood $T_\veps$ of $Z$.  It is an open neighborhood of $Z$ in the classical analytic topology and is contractible to $Z$.  The gluing of sheaves can then be performed along the punctured tubular neighborhood $T_\veps-Z$. Having this case in mind, one can view $\gtX$ as an (infinitesimal) algebraic version of a tubular neighborhood and wonder if a ``punctured formal scheme $\gtX-Z$" (or a ``generic fiber") can be meaningfully defined. It was discovered by Tate and Grothendieck that a generic fiber of a formal scheme can be defined in some cases as a non-archimedean analytic space. Today there are three different theories of such spaces: Tate's rigid spaces, Berkovich analytic spaces, and Huber's adic spaces. We choose to define $W$ as a Berkovich $k$-analytic space, where the valuation on $k$ is trivial, but other alternatives are possible and will be briefly discussed in \S\ref{othersec}. Note also that this choice is inessential on the categorical level, i.e. other choices would lead to the same category $\Coh(W)$.\footnote{When the paper was already in press and the second author reported its results at an Oberwolfach conference on Berkovich spaces, he learned from Sam Payne that an instance of the space $W$ introduced in this article was already known and used in a spectacular fashion by Amaury Thuillier in \cite{Th}. The main result \cite[Th. 4.10]{Th} states that if $k$ is perfect, $X$ is regular and proper, and $Z$ is a normal crossings divisor, then the topological type of the simplicial complex $\Delta(Z)$ depends only on $U$ and not on its compactification $U\into X$. In his proof he constructs a deformation retract $W\to\Delta(Z)$ and shows that $W$ depends only on $U$ when $X$ is proper (actually, what was used was Berkovich's space $U_\infty$ from Remark \ref{infrem}). Here are some differences in terminology and notation between this article and that of Thuillier: our tubular neighborhood (or generic fiber) $\gtX_\eta$ is denoted in \cite{Th} by $\gtX^\beth$, and our punctured tubular neighborhood $W$ is denoted in \cite{Th} by $\gtX_\eta$ and called the generic fiber of $\gtX$.}

\subsection{Overview of the paper}
We recall basic facts about Berkovich analytic spaces in \S\ref{constrsec}. In particular, in \S\ref{gensec} we recall Berkovich's generic fiber construction which associates to a {\em special} (i.e. topologically finitely presented) formal $k$-scheme $\gtX$ a $k$-analytic space $\gtX_\eta$, where $k$ is a trivially valued field. Then we establish an equivalence $\Coh(\gtX)\toisom\Coh(\gtX_\eta)$. In \S\ref{descsec} we introduce the aforementioned space $W=X_\eta-U_\eta-Z_\eta$ (where $X, Z$ and $U$ are viewed as formal schemes with the zero ideal of definition). To explain our definition of $W$ note that $X_\eta-U_\eta\toisom\gtX_\eta$ is a tubular neighborhood of $Z_\eta$, so $W$ is a punctured tubular neighborhood of $Z_\eta$. We conclude \S\ref{constrsec} with observing that since $X_\eta-U_\eta\toisom\gtX_\eta$, our main result reduces to the following key lemma:
\begin{equation}\label{eq1}
\Coh(U)\toisom\Coh(X_\eta-Z_\eta).
\end{equation}
The latter is formulated as Lemma \ref{equivlem} but its proof is postponed until \S\ref{equivsec}.

If $X$ is proper then $X_\eta\toisom X^\an$ and $Z_\eta\toisom Z^\an$, and the equivalence (\ref{eq1}) can be easily deduced from GAGA because $W\toisom X^\an-Z^\an-U_\eta=U^\an-U_\eta$. Furthermore, we see that $W$ depends only on $U$ in this case. Although it will not be used in dealing with the general case, we work out this case separately in \S\ref{tubpropsec}.

\begin{rem}\label{infrem}
The space $U_\infty=U^\an-U_\eta$ was introduced by Berkovich in 2004 in his correspondence \cite{BeD} with Drinfeld. We will outline this correspondence in \S\ref{propersec}, and will explain why $U_\infty$ can be viewed as the ``infinity" of $U$ (in Drinfeld's terminology). We would like to stress that our first construction of tubular descent was in the proper case only and used Berkovich's definition of $W=U_\infty$.
\end{rem}

In the general case, the equivalence (\ref{eq1}) is established in \S\ref{equivsec} by a straightforward computation. Then we describe the structure of $W$ in a deeper way. In particular, we show that $W$ is glued from what we call $(0,1)$-affinoid spaces (topologically they look like a product of an affinoid domain with an interval $(0,1)$), and show how such coverings can be used to work with sheaves on $W$ in an abstract algebraic way. Also, we outline alternative definitions of $W$ in \S\ref{othersec}. We conclude the paper with some examples and discussion in \S\ref{ExampleDiscuss}.

\subsection{Related descent problems}\label{relatedsec}
Fpqc descent of Grothendieck is the most standard algebro-geometric descent tool (originally appeared in a series of Bourbaki seminars, see also \cite{Des}). Note also that already in 1959 Amitsur introduced in \cite{Am} a complex useful for computing certain cohomology groups, a kind of algebraic version of the \u{C}ech complex in topology which starts with a faithfully flat map of rings.

Returning to our tubular descent problem for varieties $Z\into X$, note that if $Z$ is affine then $\gtX$ is affine, say $\gtX=\Spf(A)$, and there is a global trick that allows to pass from formal schemes to the world of schemes, where generic fibers pose no difficulties. Indeed, the equivalence $\Coh(\gtX)\toisom\Coh(A)\toisom\Coh(\calX)$, where $\calX=\Spec(A)$ and $\Coh(A)$ denotes the category of finitely generated $A$-modules, can be used to construct such a gluing by use of the usual fpqc descent. Namely, $U\coprod\calX\to X$ is an fpqc covering and, using that $U\to X$ is a monomorphism, the usual fpqc descent reduces to giving sheaves $\calF_\calX\in\Coh(\calX)$ and $\calF_U\in\Coh(U)$ together with an isomorphism $\calF_U|_\calW\toisom\calF_\calX|_\calW$, where $\calW=U\times_X\calX$, and an isomorphism $\phi$ between the pullbacks of $\calF_\calX$ to $\calX\times_X\calX$ subject to the usual triple cocycle relation. It is a result of Artin that in our case one can ignore $\phi$, and so the descent statement simplifies to the equivalence of categories $$\Coh(X)\toisom\Coh(U)\times_{\Coh(\calW)}\Coh(\calX).$$ Moreover, Artin's descent theorem \cite[2.2]{A} extends this to the following case: one only assumes that $X$ is of finite type over an {\em affine} noetherian scheme $X_0$ so that $Z$ is the preimage of a closed subset $Z_0\subseteq X_0$. Note also that \cite[4.2]{FR} generalizes Artin's descent as mentioned in the comment \cite[p. 98]{A}.

Another construction of descent is given in \cite{BL2}. It eliminates the noetherian assumption but imposes the stronger affineness condition that $X$ is affine and $U=X_f$ is its localization. Then Beauville and Laszlo establish descent for quasi-coherent modules satisfying certain $f$-regularity conditions. This required to carefully study the completion, which does not have to be flat in the non-noetherian setting. One attempt (suggested in \cite{BL2}) to extend this descent to the non-affine situation involves considering the $\calO_X$-algebra $\calO_\gtX = \varprojlim\mathcal{O}_{X}/\mathcal{I}^{m}_{Z}$, where $\mathcal{I}_Z$ is the ideal sheaf of $Z$.  One would like to take the relative Spec over $X$ of this algebra in order to produce a scheme mapping via an affine morphism to $X$.  Unfortunately,  $\calO_\gtX$ is not quasi-coherent as a sheaf of $\calO_X$-modules.  Therefore the technique of the relative Spec cannot be applied and one cannot bypass the theory of formal schemes in such a straightforward way.

Prior to writing this paper, the first author tried to salvage the situation by replacing the sheaf $\calO_\gtX$ with its coherator. This approach worked fine when $\calO_\gtX(Z)$ was large enough (e.g. for a contractible curve on a surface), but broke down in the most interesting cases, such as a curve with a positive self-intersection on $X=\bfP^2_k$. Our current approach with tubular descent does not distinguish cases, so it seems to be the most natural one, and its main achievement is in treating the cases when $Z$ has ``non-trivial global geometry in $X$", in particular, when it is not affine and, moreover, cannot be contracted to an affine variety by a birational map $X\to X'$ which is an isomorphism on $U$. See Section \ref{ExampleDiscuss} for more on these examples.

\subsection{Applications}
We now mention some possible applications of the descent we study. Harbater developed versions of patching, which is also a type of descent, in order to construct covers with given Galois groups.  A survey of patching using either formal geometry or rigid geometry can be found in \cite{H}. As we already noted, over $\bbC$ one can use tubular patching, and finding a technique which works over arbitrary fields was one of our motivations in writing this paper.

One typical historical use of descent involves an algebraic curve $X$,  a formal disc which is a formal neighborhood of a $z$ point on the curve, the complement of the point in the curve, and a punctured formal disc.  This type of descent for the map $\Spec(\widehat{\mathcal{O}_{X,z}}) \coprod (X-\{z\} ) \to X$ was proven by Beauville and Laszlo in \cite{BL1} in order to prove the Verlinde formula and study conformal blocks following work of Faltings \cite{F}.  This description has found use in the Geometric Langlands program initiated by Beilinson and Drinfeld.  Our approach replaces the picture of the formal disc mapping to  a curve with a higher dimensional one, such as a formal neighborhood of a curve mapping to a surface. We hope that such a version of descent may be helpful for extending the Geometric Langlands program to higher dimensions.  The possibility of such an extension has been mentioned by several authors; see for instance \cite{K}.

\subsection{Conventions}\label{convsec}
By $\uT$ we will often denote a tuple $(T_1\. T_n)$. Given a field $K$, by a $K$-variety we mean a $K$-scheme of finite type. By an {\em analytic field} we mean a field $K$ provided with a (multiplicative) non-archimedean real valuation $|\ |\:K\to\bfR_+$ such that $K$ is complete with respect to $|\ |$. The ring of integers of $K$ is denoted by $\Kcirc$, its maximal ideal is denoted by $K^{\circ \circ}$ and the residue field is denoted by $\tilde{K}=\Kcirc/K^{\circ \circ}$. Unlike the usual restriction (e.g. in rigid geometry), the case of trivial valuation is more than welcome -- it will be heavily used in the applications. We reserve the letter $k$ for trivially valued fields only. By a $K$-analytic space we mean a non-archimedean $K$-analytic space as defined by Berkovich in \cite{berihes}.

If $X$ is a (formal) scheme, or a $K$-analytic space then by $\Coh(X)$ we denote the category of coherent sheaves on $X$. In addition, for a noetherian ring $A$ by $\Coh(A)$ we denote the category of finitely generated $A$-modules, so $\Coh(A)\toisom\Coh(\Spec(A))$.  In this article, we will need several functors between categories of coherent sheaves on various spaces (schemes, formal schemes, and analytic spaces).  All the functors we introduce are functors of $K$-linear abelian categories which preserve the tensor structure. In particular, they preserve the property of being a free sheaf (as $\calO_X$ is the neutral object for $\otimes$). Furthermore, since a finitely presented $A$-module is flat if and only if it is locally free, all these functors also preserve the property of being locally free.  The definition of the fiber product of $K$-linear abelian categories is discussed in \ref{stackgluelem}.

\section{Some constructions of analytic spaces}\label{constrsec}
In this section we briefly recall the standard constructions that associate analytic spaces to varieties and formal schemes. More details can be found in Berkovich's works where these constructions were introduced and to which we refer below, or in the survey \cite[\S5]{survey}. Although in our applications the ground analytic field will be trivially valued, we prefer to consider the general valuation when possible. The main reason for this is that in the trivially valued case many special phenomena happen. On the one hand this will be very useful for our applications, but on the other hand it may be easier to first grasp the general picture (without ``degenerations").

\subsection{$K$-analytic spaces}

\subsubsection{Spectra of Banach rings}  Let $\mathcal{A}$ be a non-archimedean Banach ring with norm $| \ |_{\mathcal{A}}$.  Recall that a non-archimedean, bounded semivaluation $| \ |$ on $\mathcal{A}$ is a map $| \ |$  from  $\mathcal{A}$ to $\bfR_+$ which satisfies $|ab|=|a||b|, |1|=1, |0|=0, |a+b| \leq max(|a|,|b|)$ for all $a,b \in \mathcal{A}$ and such that there exists a constant $C$ such that $|a|\leq C|a|_{\mathcal{A}}$ for all $a \in \mathcal{A}$.
For any non-archimedean Banach ring $\mathcal{A}$, one defines $\mathcal{M}(\mathcal{A})$ to be the set of all non-archimedean, bounded semivaluations $| \ |$ on $\mathcal{A}$ equipped with the weakest topology making the maps from $\mathcal{M}(\mathcal{A})$ to $\bfR_+$ given by $| \ | \mapsto |f|$ continuous for all $f \in \mathcal{A}$. This construction is functorial: any bounded homomorphism $\phi$ induces a continuous map $\calM(\phi)$. Any point $x=|\ |$ of $\calM(\calA)$ induces a bounded homomorphism $\chi_x\:\calA\to\calH(x)$ to the {\em completed residue field} $\calH(x)$, which is the completion of $\Frac(\calA/\Ker(|\ |))$. Furthermore, giving a point $x$ of $\calM(x)$ is equivalent to giving an isomorphism class of bounded homomorphisms $\calA\to L$, where $L$ is an analytic $K$-field generated by the image of $\calA$ as an analytic field.

\subsubsection{$K$-affinoid algebras} \label{kAffinoid}
Let $\uT=(T_1\. T_n)$ be an $n$-tuple of elements and let $\ur$ be an $n$-tuple of positive real numbers. Then the Banach algebra $K\{\underline{r}^{-1}\underline{T}\}$ is the completion of $K[\uT]$ with respect to the norm
\[|| \sum_{i\in \mathbb{N}^{n}} a_{i} \underline{T}^{i}||_{\underline{r}} = \text{max}_{i \in \mathbb{N}^{n}}|a_{i}|\underline{r}^{i}
.\]
We call $K\{\ur^{-1}\uT\}$ a {\em Tate algebra} (though sometimes this name is reserved for the classical case when $r_i=1$). A {\it $K$-affinoid algebra} is a Banach $K$-algebra admitting an {\em admissible} surjective homomorphism from some $K\{\underline{r}^{-1}\underline{T}\}$. Admissible here means that the norm on the target is equivalent to the residue norm.

There also exists a relative version of the above construction: for any Banach ring $\calA$ we define the rings $\calA\{\underline{r}^{-1}\underline{T}\}$ using the norm as above but with $|a_i|$ replaced by $|a_i|_{\calA}$. An admissible quotient of the latter ring is called {\em $\calA$-affinoid}. This construction is transitive: if $\calB$ is $\calA$-affinoid and $\calC$ is $\calB$-affinoid then $\calC$ is $\calA$-affinoid. (To see this note that $\calA\{\underline{r}^{-1}\underline{T}\}\{\underline{s}^{-1}\underline{R}\}\toisom \calA\{\underline{r}^{-1}\underline{T},\underline{s}^{-1}\underline{R}\}$ and the functor $\calD\mapsto\calD\{\underline{s}^{-1}\underline{R}\}$ preserves admissible surjections.)

\subsubsection{$K$-analytic spaces}
The spectra $\calM(\calA)$ of $K$-affinoid algebras form a geometric category which is opposite to that of $K$-affinoid algebras. Certain morphisms of $K$-affinoid spectra are called embeddings of affinoid domains, see \cite[2.2.1]{berbook} or \cite[\S3.2]{survey}; they are analogs of open immersions of affine schemes but the important difference is that topologically such a morphism is an embedding of a compact subspace which is in general not open. The latter forces one to provide a $K$-analytic space $S$ with two different topologies: the usual topology and a Grothendieck topology on the subsets of $S$, also called a G-topology, such that both open subsets and finite unions of affinoid subdomains are G-open. In particular, the very definition of general $K$-analytic spaces is rather subtle and technical, so we refer the reader to \cite[\S1]{berihes} or \cite[\S3.3,\S4.1]{survey} for details. In these sources there is also defined a structure sheaf (sometimes written as $\calO_{S,G}$) with respect to the G-topology, while the notation $\calO_S$ is used by Berkovich to denote its restriction to the usual topology. Since we will not consider sheaves in the usual topology, to simplify notation we will write $\calO_S$ instead of $\calO_{S,G}$. The main practical tool for constructing $K$-analytic spaces is \cite[1.3.3]{berihes}, which shows how to glue such spaces.

\begin{exam}\label{KanalyticExamples}
(i) If $\ur=(r_1\.r_n)$ then the $K$-affinoid space $$E(0,\ur)=\calM(K\{\ur^{-1}\uT\})$$ is the $n$-dimensional polydisc with center at 0, of radii $r_1\.r_n$, and with coordinates $\uT=(T_1\.T_n)$.

(ii) If $\ua=(a_1\.a_n)\in K^{n}$ satisfies $|a_i|\le r_i$ then the open polydisc $D(\ua,\ur)$ of the same radii and with center at $\ua$ is defined as the open subspace of $E(0,\ur)$ given by the inequalities $|T_i-a_i|<r_i$.

(iii) The affine space $\mathbb{A}_{K}^{n,\an}=\cup_\ur E(0,\ur)$ is glued from closed $n$-dimensional polydiscs of all radii. It can be identified with the set of all real semivaluations $| \ |$  on $K[\uT]$ which are bounded on $K$ (and hence coincide with the valuation on $K$), where the topology is the weakest making $| \ | \mapsto |f|$ continuous for all $f \in K[\uT]$.
\end{exam}

\subsubsection{Quasi-nets and coherent sheaves}
Let $S$ be a $K$-analytic space. The G-open sets are called {\em analytic subdomains} of $S$. Any G-covering is a set-theoretical covering, but the main subtlety of the G-topology is that it forbids certain coverings. For example, $\bbA^1_K$ is connected in the G-topology, so although the unit disc $E$ is an affinoid subdomain of $\bbA^1_K$ and its complement $Y$ is an open subdomain in $\bbA^1_K$, the covering $\bbA^1_K=E\coprod Y$ is not a G-covering.

It is traditional to refer to G-coverings as {\em admissible coverings} (meaning that the other set-theoretic coverings are not admissible for the G-topology). Here is a very useful criterion of admissibility: a covering $S=\cup_{i\in I}S_i$ of an analytic domain by analytic domains is admissible if the $S_i$'s form a {\em quasi-net} of $S$ in the sense of \cite[\S1.1]{berihes}, i.e. for any point $s\in S$ there exists a finite subset $J\subseteq I$ such that $s\in\cap_{j\in J}S_j$ and $\cup_{j\in J}S_j$ is a neighborhood of $s$.

The G-topology is used in the definition of coherent sheaves on $S$. The following result is in virtue of \cite{berihes}, though it is not stated there explicitly. It will play a central technical role in our arguments.

\begin{lem}\label{angluelem}
If $S=\cup_{i\in I}S_i$ is an admissible covering of $S$ by analytic subdomains, $M_i$ are coherent sheaves on $S_i$, and $\phi_{ij}\:M_i|_{S_{ij}}\toisom M_j|_{S_{ij}}$ are isomorphisms of coherent sheaves on intersections that satisfy the triple cocycle condition, then there exists a unique (up to isomorphism) coherent sheaf $M$ on $S$ with isomorphisms $M|_{S_i}\toisom M_j$ that are compatible with the $\phi_{ij}$. Morphisms between coherent sheaves on $S$ can be constructed from morphisms of the $M_i$ on $S_i$ which are compatible with the $\phi_{ij}$.
\end{lem}
\begin{proof}
Let $\tau$ denote the site of the $G$-topology of $S$ and let $\tau'$ denote its full subcategory whose objects are the analytic subdomains of $S$ contained in one of $S_i$'s. We provide $\tau'$ with the induced topology, i.e. a family of morphisms in $\tau'$ is a covering if and only if it is admissible. Any admissible covering $S=\cup_{j\in J}U_j$ admits a $\tau'$-refinement $S=\cup_{i\in I,j\in J}S_i\cap U_j$, hence $\tau'$ is cofinal in $\tau$. By the locality of sheaves, any $\tau'$-sheaf (or $\tau'$-sheaf of modules) uniquely extends to a $\tau$-sheaf.

Let us show that the datum $(M_i,\phi_{ij})$ defines a $\tau'$-sheaf. If $U$ is in $\tau$ then we choose $i\in I$ with $U\subseteq S_i$ and set $M(U)=M_i(U)$. Existence of $\phi_{ij}$'s implies that this module does not depend on the choice of $i$ up to an isomorphism, and the cocycle condition ensures that $M(U)$ is well defined up to a unique isomorphism. Restriction maps $M(U)\to M(V)$ are defined in the same way, and we obtain a sheaf of $\calO_X$-modules $M$ for $\tau'$ and hence also for $\tau$. By the definition of coherence (see \cite[p. 25]{berihes}), $M$ is coherent because so are the restrictions $M|_{S_i}=M_i$ onto the elements of the quasi-net $\{S_i\}$.

The second claim of the lemma is proved similarly and simpler as no cocycle condition is involved. So, we leave the details to the interested reader.
\end{proof}

A more precise phrasing of the above facts goes as follows.  The details of this language can be found for example in 4.1.2 'the category of descent data' in \cite{V}.
\begin{lem}\label{stackgluelem}
The pairs $(\{M_i\}, \{\phi_{ij}\})$ are objects in a
k-linear, tensor, abelian category of descent data.
The restriction functor is an equivalence in the $2$-category of k-linear abelian, tensor categories
from $\Coh(S)$ to this category of descent data.
\end{lem}

\begin{rem}Of course, similar types of descent work in the categories of schemes or formal schemes using Zariski covers.  The resulting
 equivalences of categories between coherent sheaves on a space and the categories of descent data are compatible with the completion,
analytification, and generic fiber functors which we will introduce and use in this article.   In the category of (formal) schemes or analytic spaces,
the fiber product $\Coh(X_1) \times_{\Coh(X_1 \times_{X} X_2)} \Coh(X_2)$ is  the $k$-linear, abelian tensor category of descent data for the cover by two Zariski opens: $X_1,X_2$ in the case that $X$ is a (formal) scheme, or by two analytic subdomains $X_1, X_2$ in the case that $X$ is an analytic space.
\end{rem}

\subsection{Formal $\Kcirc$-schemes}  Following the reminder about formal schemes from \cite[\S5.2]{survey}, recall that for an ideal $I$ in a ring $A$, the {\it I-adic topology} is generated by the cosets $a+I^n$.  The {\it formal spectrum} $\gtX=\Spf(A)$ of an $I$-adic ring is the set of open prime ideals with the topology generated by the sets $D(f)$ with $f \in A$ and provided with the structure sheaf $\calO_\gtX$ such that $\calO_\gtX(D(f))=A_{\{f\}}$ for any $f\in A$. Here $A_{\{f\}}$ is the formal localization, i.e. it is the $I$-adic completion of $A_f$. Suppose that $K$ is trivially or discretely valued. If the valuation is non-trivial, we choose a non-zero element $\pi$ in the maximal ideal of $\Kcirc$ ($\pi=0$ when the valuation is trivial), so that $\Kcirc$ is a $\pi$-adic ring and $\Kcirc\{t_1\.t_n\}$ denotes the $\pi$-adic completion of $\Kcirc[t_1\.t_n]$. (Note that it is the subring of $\Kcirc[[t_1 \. t_n]]$ defined by the condition that for any $m \geq 0$ the number of coefficients not divisible by $\pi^m$ is finite.) By a formal $\Kcirc$-scheme we wean a formal scheme $\gtX$ with a morphism to $\Spf(\Kcirc)$. Such a scheme is {\em of topologically finite presentation} or {\em special} if it is locally of the form
\begin{equation}\label{special}
\gtX=\Spf(\Kcirc\{t_1\.t_n\}[[s_1\.s_l]]/(f_1\.f_m))
\end{equation}
If one can choose $l=0$ (i.e. the morphism $\gtX\to\Spf(\Kcirc)$ is adic) then one says that the formal $\Kcirc$-scheme is of {\em finite presentation}.   If $Z$ is a closed subvariety of a $K$-variety $X$ then we may consider the formal scheme $\gtX=\widehat{X}_{Z}$ given by the completion of $X$ along $Z$.  In this case there is a natural completion functor of $K$-linear abelian tensor categories
\begin{equation}\label{completion}\Coh(X) \to \Coh(\gtX)
\end{equation}
as can be found for instance in Chapter II, section 9 of \cite{Ha}.

\begin{rem}
Berkovich introduced finitely presented formal schemes in \cite[\S1]{bervanish1} without any restriction on $K$. However, if $K$ is not trivially or discretely valued then the rings $\Kcirc[[t_1\.t_n]]$ behave wildly (e.g. they possess non-closed ideals). In particular, even the basic theory of formal schemes for such rings is not developed. Special formal schemes were introduced in \cite[\S1]{bervanish2} only when $K$ is trivially or discretely valued, in particular, these formal schemes are noetherian. We will stick to this case as well, although it seems that the restriction on $K$ is technical and can be removed after developing enough foundations.
\end{rem}

Let $\mathfrak{I}$ be an ideal of definition of a special formal $\Kcirc$-scheme $\gtX$. The radical $\gtJ=\sqrt{\gtI}$ does not depend on the choice of $\gtI$ and the ringed space $(\mathfrak{X},\mathcal{O}_{\mathfrak{X}}/\gtJ)$ is a reduced scheme of finite type over $\Kcirc$ denoted $\mathfrak{X}_{s}$.

\subsection{Generic fiber and reduction map}\label{gensec}
To any special formal $\Kcirc$-scheme $\gtX$ one can associate a $K$-analytic space $\gtX_\eta$ called the {\em generic fiber} of $\gtX$. In the framework of classical rigid geometry (when the valuation is non-trivial and the generic fiber is constructed as a rigid space over $K$) this construction is due to Raynaud in the finite presentation case, and to Berthelot in general. Analogous constructions using analytic spaces were developed by Berkovich in \cite[\S1]{bervanish1} and \cite[\S1]{bervanish2}. In addition, there is a natural {\em reduction map} $\pi_{\gtX}:\gtX_\eta\to\gtX_s$ which is anti-continuous in the sense that the preimage of an open subspace is closed.

\subsubsection{Affine case}\label{genaffsec}
We discuss the affine case first, so let $\gtX=\Spf(A)$ and choose a presentation
$$\Kcirc\{t_1\.t_n\}[[s_1\.s_l]]/(f_1\. f_m)\toisom A.$$
Consider the product $E(0,1)^n\times D(0,1)^l$ of $n$ closed unit discs with coordinates $t_i$ and $l$ open unit discs with coordinates $s_j$ and define $\gtX_\eta$ as the closed subspace of $E(0,1)^n\times D(0,1)^l$ given by the vanishing of $f_1\. f_m$. It is easy to see that $\gtX_\eta$ is independent of the choice of the presentation, so it is well defined.

If $\gtX$ is of finite presentation then one can choose a presentation as in (\ref{special}) with $l=0$. In particular, $\gtX_\eta=\calM(A_\eta)$, where $A_\eta=A\otimes_{\Kcirc}K\toisom K\{t_1\.t_n\}/(f_1\.f_m)$, and so $\gtX_\eta$ is affinoid in this case.

Any point $x\in\gtX_\eta$ corresponds to a bounded character $\chi_x\:\calO(\gtX_\eta)\to\calH(x)$ and restricting it on $A$ we obtain a homomorphism $A\to\calH(x)$. The reduction homomorphism $A/\Rad(\pi A+s_1A+\dots+s_lA)\to\wHx$ gives a point $\tilx\in\gtX_s$, and the correspondence $x\mapsto\tilx$ defines the map $\pi_\gtX$.

\begin{rem}
(i) On the level of sets one can describe $\gtX_\eta$ as the set of all bounded semivaluations on $\Kcirc\{t_1\.t_n\}[[s_1\.s_l]]$ that extend the valuation on $\Kcirc$ and satisfy $|t_i|\le 1$, $|s_j|<1$, and $|f_k|=0$. Equivalently, these are bounded semivaluations on $A$ that extend the valuation on $\Kcirc$ and satisfy $|t_i|\le 1$ and $|s_j|<1$. If $\gtX_\eta$ is finitely presented then this set also coincides with the set of bounded semivaluations on the $K$-affinoid algebra $A_\eta$ that satisfy $|t_i|\le 1$ and $|s_j|<1$ (a bounded semivaluation on $A_\eta$ automatically agrees with the valuation on $K$).

(ii) In very low dimensions (analytic curves over any $K$, or analytic surfaces over a trivially valued field) one can describe $\gtX_\eta$ rather explicitly but the general case seems to be too complicated.
\end{rem}

\subsubsection{Gluing}
The affine generic fiber construction from \S\ref{genaffsec} is functorial and takes open immersions to closed subdomain embeddings. It follows that the construction globalizes to separated special formal schemes: one covers such $\gtX$ by open affine formal subschemes $\gtX_i$ and glues together the spaces $(\gtX_i)_\eta$ along the analytic subdomains $(\gtX_{ij})_\eta$. Applying the same construction again, one glues together the generic fiber of a general special formal scheme from the generic fibers of its separated formal subschemes. If $\gtX$ is finitely presented then $\gtX_\eta$ is glued from compact domains along compact domains, and hence is compact. (It is easy to show that the converse is also true, but we will not need that.)

The reduction maps glue as well, so for any special formal scheme $\gtX$ one obtains an analytic space $\gtX_\eta$ with an anti-continuous reduction map $\pi_\gtX\:\gtX_\eta\to\gtX_s$. Recall that by \cite[1.3]{bervanish2}, if $Y\into \gtX_s$ is a closed subscheme and $\gtY$ is the formal completion of $\gtX$ along $Y$ then the natural morphism
\begin{equation}\label{GenericFormalSub}\gtY_\eta\to\pi_\gtX^{-1}(Y)
\end{equation} is an isomorphism.

\subsubsection{Generic fiber of modules}\label{genfibsec}
The generic fiber functor $\gtX\mapsto\gtX_\eta$ extends to the categories of coherent sheaves as follows. Let $\gtX=\Spf(A)$ be as in \S\ref{genaffsec} and let $M$ be a finite $A$-module. For any $0<r<1$ set $A_r=k\{\ut,r^{-1}\us\}/(\uf)$ and $M_r=M\otimes_A A_r$. Since $M_r$ is a finite $A_r$ module, it defines a coherent module on $X_r=\calM(A_r)$ (we use here that by \cite[2.1.9]{berbook} $M_r$ possesses a unique structure of Banach $A_r$-module). The obtained modules on the spaces $X_r$ agree, hence we obtain a module $M_\eta$ on $\gtX_\eta=\cup_{r<1}X_r$. The construction easily globalizes, so for any special $\gtX$ we obtain a generic fiber functor \begin{equation}\label{gff}\eta\:\Coh(\gtX)\to\Coh(\gtX_\eta).\end{equation}

\subsubsection{Trivially valued ground field}
If the ground field $k=K$ is trivially valued then the generic fiber functor possesses some nice properties that do not hold otherwise.

\begin{lem}\label{GenCat} Let $k$ be a trivially valued field and let $\gtX$ be a special formal $k$-scheme. Then the generic fiber functor $\Coh(\gtX)\toisom\Coh(\gtX_\eta)$ is an equivalence of $k$-linear abelian categories.
\end{lem}
\begin{proof}
In view of Lemma \ref{angluelem}, it suffices to consider the affine case, so assume that $\gtX=\Spf(A)$ where $A=k\{\ut\}[[\us]]/(\uf)$. Note that $k\{\ut\}=k[\ut]$ with the trivial norm. Since $A$ is a noetherian adic ring, $\Coh(\gtX)$ is equivalent to the category $\Coh(A)$ of finite $A$-modules. The generic fiber $\gtX_\eta$ is the filtered union of the affinoid domains $V_r=\calM(A_r)$, where $0<r<1$ and $A_r$ is as in \S\ref{genfibsec}. For any such $r$, we have that $\phi_r\:A\toisom A_r$ as a $k$-algebra (though not as a Banach algebra), in particular, $\Coh(A)\toisom\Coh(A_r)$. By \cite[2.1.9]{berbook}, the latter is equivalent to the category of
finite Banach $A_r$-modules, which is equivalent to $\Coh(V_r)$ by definition. The latter two equivalences are canonical (i.e. defined up to a unique isomorphism of functors), so $\phi_r$ determines an equivalence $\psi_r\:\Coh(A)\toisom\Coh(V_r)$ up to a unique isomorphism of functors.

The bijective homomorphisms $\phi_{rs}\:A_r\to A_s$ corresponding to the inclusions $V_s\into V_r$ for $0<s\le r<1$ induce equivalences $\psi_{rs}\:\Coh(V_r)\toisom\Coh(V_s)$. Since $\phi_s=\phi_{rs}\circ\phi_r$ we obtain that $\psi_s$ is canonically isomorphic to $\psi_{rs}\circ\psi_r$. Therefore the equivalences $\psi_{rs}$ are compatible: $\psi_{rs}\toisom\psi_{ts}\circ\psi_{rt}$. It follows that $\Coh(\gtX_\eta)\toisom\Coh(V_r)$ and $\Coh(\gtX)\toisom\Coh(\gtX_\eta)$, as claimed.
\end{proof}

\begin{rem}
(i) Each base change $\psi_{rs}$ is given by the functor $M_r\mapsto M_s=M_r\wtimes_{A_r}A_s$. Although $\phi_{rs}$ is bijective, it is not an isomorphism of Banach rings when $s<r$ (e.g. its inverse is unbounded). In particular, $M_r\toisom M_s$ as an $A$-module, but the norm of $M_r$ may differ from that of $M_s$.

(ii) In the affine case, the generic fiber functor preserves $M$ as a module but replaces its adic topology with a Frechet module structure.
\end{rem}

\subsection{Tubular descent}\label{descsec}
Once the generic fiber functor is introduced, we can define $W$ and formulate the main result on tubular descent. This will be done in this section.

\subsubsection{Notation}\label{notsec}
We fix the following notation until the end of the paper: $k$ is a ground field, which we provide with the trivial valuation, $X$ is a $k$-variety as defined in \ref{convsec}, $Z\into X$ is a closed subvariety, $\gtX=\hatX_Z$ is the formal completion of $X$ along $Z$, and $U\subset X$ is the complement of $Z$. Our aim is to glue the category $\Coh(X)$ from $\Coh(U)$ and $\Coh(\gtX)$. The gluing will be along the category $\Coh(W)$, where $W$ is defined below.

\subsubsection{The definition of $W$}\label{Wsec}
The generic fiber construction provides a compact analytic space $X_\eta$ with compact subdomain $U_\eta$ and closed subspace $Z_\eta$, and we set $V=X_\eta-Z_\eta$ and $W=X_\eta-U_\eta-Z_\eta$. Note that by applying (\ref{GenericFormalSub}) to the case $Y=Z=\gtX_s$ we have $X_\eta-U_\eta=\pi_X^{-1}(Z)=\gtX_\eta$, and hence $Z_\eta\subset\gtX_\eta$ is disjoint from $U_\eta$.

\begin{rem}\label{Wrem}
(i) Unlike the algebraic case, both $U_\eta$ and $Z_\eta$ are closed, so the $k$-analytic space $W=X_\eta-U_\eta-Z_\eta$ is usually non-empty. Actually, $W=\emptyset$ if and only if $U_\eta$ and $Z_\eta$ are each individually both open and closed, but then also $U$ and $Z$ are each individually both open and closed, and so $X$ is not connected. This explains why we have enough ``meat" to perform the desired descent in the analytic category.

(ii) Note that $\gtX_\eta=X_\eta-U_\eta$ is a tubular neighborhood of $Z_\eta$ in $X$ and $W=\gtX_\eta-Z_\eta$ is a punctured tubular neighborhood of $Z_\eta$. So, one can view $W$ as a $k$-analytic version of a punctured tubular neighborhood of $Z$. Note also that $W$ is determined by the formal scheme $\gtX$, namely it coincides with $\gtX_\infty=\gtX_\eta-(\gtX_s)_\eta$ (we cannot reconstruct the non-reduced structure of $Z$ from the formal completion $\gtX$, but $\gtX_s$ is the reduction of $Z$, and this suffices to define $W$). Thus, although one cannot reasonably define the ``punctured formal scheme" $\gtX-\gtX_s$ in the category of formal schemes, $\gtX_\infty$ provides a realization of such an object as an analytic space.

(iii) We already remarked in \S\ref{relatedsec} that if $\gtX=\Spf(A)$ is affine then we can replace it with $\calX=\Spec(A)$ in the formulation of the descent using $\calW=\calX\times_XU$ for gluing. One cannot globalize this construction in a naive way: if one covers $X$ with affine open subschemes $X_i$ and constructs schemes $\calW_i$ as above, then one cannot glue them as schemes (or formal schemes). We will show in Lemmas \ref{Wlem} and \ref{Wgenlem} that $W$ is a $k$-analytic version of such a gluing.
\end{rem}

Also, let us describe $W$ in the affine case. In this example and in the sequel we follow the convention that if a finitely generated $k$-algebra is considered as a Banach $k$-algebra then it is provided with the trivial norm.

\begin{exam}\label{spacesexam}
Assume that $X=\Spec(A)$ is affine, and let $I=(f_1\.f_n)$ be the ideal defining $Z$. Then $X_\eta=\calM(A)$ (where $|\ |_A$ is trivial), so $X_\eta$ consists of all bounded semivaluations of $A$, and we have the following description of the relevant subspaces: $\gtX_\eta$ is defined by the conditions $|f_i|<1$ (for $1\le i\le n$), $Z_\eta$ is defined by the conditions $|f_i|=0$, $U_\eta$ is defined by the conditions $|f_i|=1$, and $W$ is defined by the conditions $|f_i|<1$ for each $1\le i\le n$ and $|f_i|>0$ for some $i$.
\end{exam}

\subsubsection{The key lemma}
In order to formulate our main descent result we will need the following lemma. Its proof is rather computational, and we postpone it until \S\ref{equivsec} (with no circular reasoning happening).

\begin{lem}\label{equivlem}
Assume that $k$ is a trivially valued field and $X$ is a $k$-variety with open subvariety $U$ and $Z=X-U$, and set $V=X_\eta-Z_\eta$. Then the restriction functor $\Coh(V)\to\Coh(U_\eta)$ is an equivalence.
\end{lem}

\subsubsection{The main result}
In the following theorem, we will use the generic fiber functor $\eta$ defined in (\ref{gff}) and the completion functor (\ref{completion}). We also use three restriction functors ($k$-linear functors of tensor abelian categories): (i) the functor $\Res:\Coh(\gtX_\eta) \to \Coh(W)$ given by combining (\ref{GenericFormalSub}) in the case that $Y=Z=\gtX_s$ with the restriction of coherent sheaves from $X_\eta-U_\eta$ to $W$, (ii) $\Res:\Coh(V)\to \Coh(W)$ given by restriction, and (iii) $\Res:\Coh(V)\to \Coh(U_\eta)$ which is also given by restriction. By Lemma \ref{equivlem} the latter is an equivalence, hence its inverse in the following theorem is well defined up to a unique isomorphism of functors.

\begin{theor}\label{MainThm}
Assume that $k$ is a trivially valued field, $X$ is a $k$-variety, $Z\into X$ is a closed subvariety, $U=X-Z$ is the complement, and $\gtX=\hatX_Z$ is the formal completion along $Z$. Consider the $k$-analytic varieties $X_\eta$, $U_\eta$, $Z_\eta$, and set $V=X_\eta-Z_\eta$ and $W=X_\eta-U_\eta-Z_\eta$. Consider the functors
$$\Coh(\gtX)\stackrel{\eta}{\to}\Coh(\gtX_\eta)\stackrel{\Res}{\to}\Coh(W)$$ and \begin{equation}\label{chain} \Coh(U)\stackrel{\eta}{\to}\Coh(U_\eta)\stackrel{\Res^{-1}}{\toisom}\Coh(V)\stackrel{\Res}{\to}\Coh(W)\end{equation} where $\Res$ denote the restriction functors. Then the restriction $\Coh(X)\to\Coh(U)$ and the completion functor $\Coh(X)\to\Coh(\gtX)$ induce an equivalence of $k$-linear tensor abelian categories
\begin{equation}\label{FibProd}\Coh(X)\toisom\Coh(U)\times_{\Coh(W)}\Coh(\gtX).\end{equation}
\end{theor}

Note that the functors $\eta$ above are also equivalences, but we do not need this to formulate the theorem.

\begin{proof}
By \S\ref{Wsec}, $X_\eta$ is covered by $V=X_\eta-Z_\eta$ and $\gtX_\eta=X_\eta-U_\eta$. Being generic fibers of finitely presented formal $k$-schemes, the spaces $Z_\eta$ and $U_\eta$ are compact, so their complements are open. Thus, this covering is admissible and by Lemma \ref{angluelem} there is an equivalence $$\Coh(X_\eta)\toisom\Coh(V)\times_{\Coh(W)}\Coh(\gtX_\eta).$$
Since $\Coh(X)\toisom\Coh(X_\eta)$, $\Coh(\gtX)\toisom\Coh(\gtX_\eta)$, and $\Coh(U)\toisom\Coh(U_\eta)$ by Lemma \ref{GenCat}, and $\Coh(V)\toisom\Coh(U_\eta)$ by Lemma \ref{equivlem}, we obtain the assertion of the theorem.
\end{proof}

\begin{rem}
As we remarked in \S\ref{convsec}, the equivalence in Theorem \ref{MainThm} preserves locally free sheaves and thereby we get descent for groupoids of vector bundles.  There is a similar statement for the endomorphisms of any vector bundle on $X$ and by considering the trivial bundle we have an isomorphism of rings $\mathcal{O}(X) \cong \mathcal{O}(U) \times_{\mathcal{O}(W)} \mathcal{O}(\gtX)$.  By properly interpreting infinite dimensional algebraic varieties, the stack of rank $r$ vector bundles on $X$ (over the site of schemes over $k$ of finite type) trivializing on $U$ and on $\gtX$ can be identified with the quotient stack $[GL_{r}(\mathcal{O}(U)) \backslash GL_{r}(\mathcal{O}(W))/GL_{r}(\mathcal{O}(\gtX))]$, as in \cite{F, BL1}.   We give a specific example in \ref{ProjSpace}.
\end{rem}

\section{Tubular descent: the proper case}\label{tubpropsec}
We will show in this section that if $X$ is proper then the key lemma and tubular descent easily follow from a standard tool of Berkovich geometry -- the non-archimedean GAGA theory. We will first recall analytification of varieties and GAGA, and its relation to the generic fiber construction. Then, the descent will be constructed in a couple of lines. The material of this section may be instructive. On the other hand, it will not be used when dealing with the general case, so the reader can skip directly to \S\ref{gencase}.

\subsection{Analytification}

\subsubsection{Affine case}
Assume that $K$ is an analytic field. Then to any $K$-variety $Y$ one can associate a $K$-analytic space $Y^\an$
called the {\em analytification} of $Y$. If $Y=\Spec(A)$ for $A=K[T_1\. T_n]/(f_1\. f_m)$ then $Y^\an$ is the analytic subspace of the affine space $\mathbb{A}^{n,\an}_K$ with coordinates $T_1\. T_n$ defined by the vanishing of $f_1\. f_m$. This construction is independent of the presentation of $A$, functorial, and compatible with open immersions.

\begin{exam}
For instance, if $A=K[\underline{T}]/I$ is a finitely generated $K$-algebra and $g \in A$ then $(\Spec(A_{g}))^\an \subset \Spec(A)^\an$ consists of the subset of non-archimedean, bounded semivaluations on $K\{\underline{T}\}$ which vanish on $I$ but not on $g$.
\end{exam}

\subsubsection{Gluing}
Since analytification respects open immersions, it extends to a functor $Y\mapsto Y^\an$ from the category of $K$-varieties to the category of $K$-analytic spaces. Explicitly, if $Y$ is separated and $Y=\cup_i Y_i$ is an open affine covering then $Y^\an$ is glued from the spaces $Y_i^\an$ along open subspaces $Y_{ij}^\an$. This extends the analytification functor to the category of all separated varieties, and applying the same procedure again we extend it further to the category of all $K$-varieties. Note that $Y^\an$ is a {\em good} $K$-analytic space, i.e. any of its points is contained in an affinoid neighborhood (good spaces are the spaces introduced in \cite{berbook}).

\subsubsection{Analytification of modules}
The analytification $Y^\an$ of $Y$ represents morphisms of locally ringed spaces $(S,\calO_S)\to(Y,\calO_Y)$ with source a good $K$-analytic space, see \cite[\S3.4]{berbook}. In particular, there is the universal analytification morphism $\alp_Y\:(Y^\an,\calO_{Y^\an})\to(Y,\calO_Y)$ (denoted by $\pi_Y$ in loc.cit.) that induces an analytification functor $\Coh(Y)\to\Coh(Y^\an)$ via the rule $\calF\mapsto\calF^\an=\alp_Y^*(\calF)=\alp_Y^{-1}(\calF)\otimes_{\alp_Y^{-1}(\calO_Y)}\calO_{Y^\an}$.

\subsubsection{GAGA}
All classical complex analytic GAGA results have non-archimedean counterparts, as was shown by Berkovich in \cite[\S3.4]{berbook}. In particular, $Y$ is proper (resp. separated) if and only if $Y^\an$ is compact (resp. Hausdorff) and for a proper $Y$ the analytification functor $\Coh(Y)\to\Coh(Y^\an)$ is an equivalence.

\subsubsection{Trivial valuation}\label{GAGAtriv}
If $k$ is trivially valued then many results of GAGA extend to non-proper varieties, \cite[\S3.5]{berbook}. In particular, in this case the functor $Y\mapsto Y^\an$ is fully faithful and for any $k$-variety $Y$ the functor $\Coh(Y)\to\Coh(Y^\an)$ is an equivalence. It follows that for any $k$-variety $Y$ there is an isomorphism of rings $\calO_Y(Y)\toisom\calO_{Y^\an}(Y^\an)$ (and actually, one has that $H^i(\calF)\toisom H^i(\calF^\an)$ for any coherent $\calO_Y$-module $\calF$). If $K$ is not trivially valued then this fails already for $\bfA^1_K$, as there exist non-polynomial analytic functions on $\bfA^{1,\an}_K$.

\subsection{A relation between analytification and generic fiber functors}\label{relationsec}
The aim of this section is to discuss a relationship between the analytification and the generic fiber functors observed by Berkovich in \cite[\S5]{bervanish1}.

\subsubsection{Two functors from $\Kcirc$-schemes to $K$-analytic spaces} Let $Y$ be a flat $\Kcirc$-scheme of finite presentation. Then there are two functorial constructions that associate to $Y$ a $K$-analytic space. First, we can consider the generic fiber $Y\otimes_{\Kcirc}K$ of the morphism $Y\to\Spec(\Kcirc)$, which is a $K$-variety, and then we can analytify $Y\otimes_{\Kcirc}K$. This construction produces a functor $\calF(Y)=(Y\otimes_{\Kcirc}K)^\an$. The second construction first completes $Y$ to a formal $\Kcirc$-scheme $\hatY$ and then takes the generic fiber. The output is a functor $\calC(Y)=\hatY_\eta$.

\subsubsection{The comparison transformation}\label{comparsec}
Assume that $Y$ is affine, say $Y=\Spec(A)$ with $A=\Kcirc[t_1\.t_n]/(f_1\.f_m)$. Then $\calC(Y)=\hatY_\eta$ is the affinoid subdomain of $\calF(Y)=(Y\otimes_{\Kcirc}K)^\an$ defined by the inequalities $|t_i|\le 1$. Indeed, for $T=\bfA^n_{\Kcirc}=\Spec(\Kcirc[t_1\.t_n])$ this is clear as $\calF(T)=\bfA^{n,\an}_K$ and $\calC(T)=E(0,1)^n$ is the unit polydisc. And in general, $\calF(Y)$ and $\calC(Y)$ are the closed subspaces of $\calF(T)$ and $\calC(T)$, respectively, given by the vanishing of $f_j$'s.

The embedding morphisms $i_Y\:\calC(Y)\to\calF(Y)$ are compatible with localizations, hence we obtain a natural transformation of functors $i\:\calC\to\calF$. If $Y$ is separated then $i_Y$ is an embedding of an analytic subdomain by \cite[\S5]{bervanish1}. For a proper $Y$ both $Y_\eta$ and $Y^\an$ are proper, hence $Y_\eta$ is a connected component of $Y^\an$. Using that the connected components of $Y$ are in one-to-one correspondence with connected components of both $Y_\eta$ and $Y^\an$, we obtain that $i_Y\: Y_\eta\toisom Y^\an$.

\subsubsection{Trivial valuation}
If $k=K$ is trivially valued then $k=\kcirc$ and various intermediate objects in the above construction coincide: $Y\otimes_{\kcirc}k=Y=\hatY$ can be viewed both as a $k$-variety and a special formal $\kcirc$-scheme. In particular, analytification and generic fiber provide two functors from $k$-varieties to $k$-analytic spaces and $i$ is a natural transformation between them.

\begin{lem}\label{comparlem}
Assume that $k$ is trivially valued and $Y$ is a $k$-variety. Then $i_Y\: Y_\eta\to Y^\an$ is an isomorphism (resp. an embedding of a subdomain) if and only if $Y$ is proper (resp. separated).
\end{lem}
\begin{proof}
The converse implications were established above. Conversely, if $Y$ is not proper (resp. separated) then $Y^\an$ is not compact (resp. Hausdorff) by GAGA. Since $Y_\eta$ is compact, $i_Y$ is not an isomorphism (resp. embedding of a subdomain).
\end{proof}

\subsection{Application to tubular descent}\label{propdescsec}
Let $k,X,U,Z,\gtX$ be as defined in \S\ref{notsec}. Let also $W=X_\eta-U_\eta-Z_\eta$ and $V=X_\eta-Z_\eta$, as earlier. Throughout \S\ref{propdescsec} we also assume that $X$ is proper.

\subsubsection{The main result}
Recall that $\Coh(U_\eta)\toisom\Coh(U)$ by Lemma \ref{GenCat}. Since $X$ and $Z$ are proper we have that $X_\eta\toisom X^\an$ and $Z_\eta\toisom Z^\an$ by Lemma \ref{comparlem}. Hence $V$ coincides with $U^\an=X^\an-Z^\an$, and using GAGA we obtain that $\Coh(U)\toisom\Coh(U^\an)=\Coh(V)$. This, proves the key Lemma \ref{equivlem} for a proper $X$ and completes the proof of Theorem \ref{MainThm} in this case.

\subsubsection{Infinite fiber of $U$}\label{propersec}
As we have just seen, for a proper $X$ the space $W$ coincides with the space $U_\infty=U^\an-U_\eta$. In particular, $W$ depends only on $U$. We suggest to call $U_\infty$ the infinite fiber of $U$.

The space $U_\infty$ was introduced by Berkovich in a private correspondence with Drinfeld in 2004. In his letter \cite{BeD}, Drinfeld asked if one can associate to a $k$-variety $U$ a geometric object (presumably a non-archimedean space) called ``infinity" of $U$ that might be considered as a sort of universal compactification of $U$. Drinfeld's suggestion was to pick up any compactification $U\into X$, to consider the completion $\gtX$ at infinity (i.e. the completion along $Z=X-U$), and to remove the closed fiber. Drinfeld's expectation was that such an infinity ``$\gtX-\gtX_s$" exists as a slightly generalized analytic space, is independent of the compactification, can be constructed by gluing the spaces $Y_i=\calM(A_i((f_i)))$ whenever $Z$ is Cartier (here $A((f_i))$ is as in Definition \ref{01def}(i)), and can be constructed in general from the pair $(\Bl_Z(X),Z\times_X\Bl_Z(X))$ instead of $(X,Z)$ (see \S\ref{invarsec} below). The subtle point pointed out by Drinfeld is that $Y_i$ has no canonical field of definition (as $k((f_i))$ is just one of many choices), so one should extend the category considered by Berkovich. We will outline in \S\ref{kspacesec} how this can be formalized.

In his answer, Berkovich suggested an approach that works perfectly well within the usual category of $k$-analytic spaces. He introduced $U_\infty$ (under the name ``infinity of $U$") and showed that, indeed, it depends only on $U$. On the topological level, the space $U_\infty$ is obtained from Drinfeld's suggestion by taking the product with $(0,1)$.

\section{Tubular descent: the general case}\label{gencase}

\subsection{Proof of the key lemma}\label{equivsec}
We already saw in the proof of Lemma \ref{GenCat} that certain $k$-analytic domains $Y'\into Y$ induce equivalences $\Coh(Y)\to\Coh(Y')$. Naturally, our study of sheaves on $V=X_\eta-Z_\eta$ will be based on finding some more domains that possess the above property.

\begin{proof}[Proof of Lemma \ref{equivlem}]
Step 1. {\it The lemma holds true when $X=\Spec(A)$ and $U=\Spec(A_f)$.} In this case, $V$ is the union of affinoid domains $V_r=\calM(A\{rf^{-1}\})$ with $0<r\le 1$. Clearly, $V_s\subseteq V_r$ when $0<r\le s\le 1$ and, although the norm of the Banach ring $A\{rf^{-1}\}$ depends on $r$, we have that $$A\{rf^{-1}\}=A\{rT\}/(Tf-1)=A[T]/(Tf-1)=A_f$$ on the level of $k$-algebras.

\begin{rem}
It is easy to see that if $f$ is non-invertible then $V_s\varsubsetneq V_r$ for $r<s$, and hence the norms on $A\{rf^{-1}\}$ and $A\{sf^{-1}\}$ are not even equivalent. This claim will not be used, so we only hint at the argument: $X_\eta$ contains a point $|\ |_x$ with $0<|f|_x<1$ hence an appropriate power $|\ |_y=|\ |_x^t$ satisfies $r<|f|_y<s$. So, $y\in V_r\setminus V_s$.
\end{rem}

From the above chain of equalities we have canonical equivalences $\Coh(A_f)\toisom\Coh(V_r)$ for any $0<r<1$ which are compatible with the restrictions $\phi_{r,s}\:\Coh(V_r)\to\Coh(V_s)$ for $0<r\le s\le 1$. In particular, these restrictions $\Coh(V_1)\to\Coh(V_r)$ are equivalences compatible with the equivalences $\phi_{r,s}$ and hence $\Coh(V)\toisom\Coh(V_1)$. It remains to observe that for $r=1$ we have that $V_r=\calM(A\{f^{-1}\})=\calM(A_f)$ where $A_f$ is provided with the trivial norm, hence $V_1=U_\eta$.

Step 2. {\it The lemma holds when $X=\Spec(A)$.} Let $Z$ be given by an ideal $I=(f_1\. f_n)$. Then $Z$ is the intersection of $Z_i=\Spec(A/(f_i))$, $U$ is the union of $U_i=\Spec(A_{f_i})$, and $V$ is the union of $V_i=X_\eta-(Z_i)_\eta$. Let $U_{ij}$ and $V_{ij}$ (resp. $U_{ijk}$ and $V_{ijk}$) denote the double (resp. triple) intersections, and let $Z_{ij}=Z_i\cup Z_j$. Note that $V_{ij}=X_\eta-(Z_{ij})_\eta$ because $(Z_{ij})_\eta=(Z_i)_\eta\cup(Z_j)_\eta$, and similarly for the triple intersections. By Step 1 we have equivalences $\Coh(V_i)\toisom\Coh((U_i)_\eta)$, $\Coh(V_{ij})\toisom\Coh((U_{ij})_\eta)$, and $\Coh(V_{ijk})\toisom\Coh((U_{ijk})_\eta)$ that are, clearly, compatible with the restrictions between $V$'s and $U_\eta$'s. Therefore, Lemma \ref{stackgluelem} implies that these equivalences glue to the equivalence $\Coh(V)\toisom\Coh(U_\eta)$.

Step 3. {\it The lemma holds when $X$ is separated.} In this case we find an affine open covering $X=\cup_{i\in I} X_i$ and note that for any non-empty $J\subseteq I$ the intersection $X_J=\cap_{j\in J}X_j$ is affine (as earlier, we will only need double and triple intersections). Set $U_J=U\times_XX_J$, $Z_J=Z\times_XX_J$, and $V_J=(X_J)_\eta-(Z_J)_\eta$. Note that $(Z_J)_\eta=Z_\eta\cap(X_J)_\eta$ and hence $V_J=V\cap(X_J)_\eta$. In particular, $V_{J\cup J'}=V_J\cap V_{J'}$ and $U_{J\cup J'}=U_J\cap U_{J'}$. By Step 2, $\Coh(U_J)\toisom\Coh(V_J)$ for any non-empty $J$, and clearly these equivalences are compatible with the restrictions. Therefore, they glue to an equivalence $\Coh(U)\toisom\Coh(V)$.

Step 4. {\it The general case.} We cover $X$ by separated open subschemes $X_i$ and repeat the argument of Step 3, but with the equivalence from Step 3 used as an input.
\end{proof}

We saw that the space $W$ plays the crucial role in the tubular descent. The aim of the next three sections is to study $W$ in more detail.

\subsection{A closer look at $W$: affine Cartier case}\label{affcasesec}
Let us consider the case when $X=\Spec(A)$ and $U=\Spec(A_f)$, in particular, $Z=V(f)$ and $\gtX=\Spf(\hatA)$ where $\hatA$ is the $(f)$-adic completion of $A$. Actually, this is the simplest case that was established first in our proof of Lemma \ref{equivlem}. In this case we can describe $W$ very concretely. In the sequel, for any $0<r<1$ let $K_r$ denote the field $k((f))$ provided with the valuation $|\ |_r$ which is trivial on $k[[f]]^\times$ and satisfies $|f|_r=r$.

\begin{lem}\label{affcaselem}
Assume that $X=\Spec(A)$ and $U=\Spec(A_f)$. Let $\hatA$ denote the $(f)$-adic completion of $A$, let $B=\hatA_f=\hatA[f^{-1}]$, and let $\calB_r$ be the normed $k$-algebra $(B,\|\ \|_r)$, where $\|x\|_r=r^{-n}$ for the minimal $n\in\bfZ$ such that $f^nx\in\hatA$.

(i) $W$ is the subdomain of $X_\eta$ given by the condition $0<|f|<1$.

(ii) $\calO_{X_\eta}(W)\toisom B$ and $\Coh(B)\toisom\Coh(W)$.

(iii) Each point of $W$ induces a semivaluation on $B$ that restricts to a valuation $|\ |_r$ on $k((f))\subseteq B$, so a map $W\to(0,1)$ arises. If $0<r<1$ is fixed, then $\calB_r$ is a $K_r$-affinoid algebra and the fiber of $W$ over $r$ is homeomorphic to the $K_r$-affinoid space $W_r=\calM(\calB_r)$. If we identify the topological space $W_r$ with other fibers $W_s$ by sending a semivaluation to its $\frac sr$-powers then the fiber homeomorphisms glue to a homeomorphism $W\toisom W_r\times(0,1)$.
\end{lem}
\begin{proof}
(i) is a particular case of Example \ref{spacesexam}. Next, note that $W=\cup_{0<r\le s<1}W_{[r,s]}$ where
$$W_{[r,s]}=X_\eta\{r\le|f|\le s\}=\calM(A\{s^{-1}x,rx^{-1}\}/(x-f))$$ In particular, $\calO_{X_\eta}(W_{[r,s]})\toisom A((t))/(t-f)\toisom B$ for any choice of $0<r\le s<1$, and hence $\Coh(B)\toisom\Coh(W_{[r,s]})$. This implies (ii). Finally, taking $s=r$ we obtain that $W_r$ is as asserted in (iii), and the remaining claim of (iii) follows.
\end{proof}

Let us make a few comments concerning the lemma.

\begin{rem}
(i) A naive algebraic attempt to remove the special fiber from $\gtX$ is to formally localize $\hatA$ at the topologically nilpotent element $f$. Such operation (similarly to the usual localization at a nilpotent element) produces the zero ring $\hatA_{\{f\}}=0$. So, as one should expect, $\gtX-\gtX_s=\emptyset$.

(ii) The correct way to invert $f$ is to consider the usual localization $B=\hatA_f$. One can then induce a topology on $B$ by declaring $\hatA$ to be open. Such rings (with an open {\em ring of definition} whose topology is adic) were called f-adic by R. Huber. Actually, they form a very natural class of topological rings. For example, the natural topology on a complete height one valuation field makes it into a f-adic ring. We have shown that in the affine situation described in the lemma $\calO_{X_\eta}(W)$ is f-adic.
\end{rem}

\begin{defin}\label{01def}
(i) If $A$ is a finitely generated $k$-algebra and $f\in A$ then by $A((f))$ we denote the f-adic ring $\hatA_f$, where $\hatA$ is the $(f)$-adic completion of $A$. By $A((f))_r$ we denote the $K_r$-affinoid algebra whose underlying topological algebra is $A((f))$.

(ii) A $k$-analytic space $W$ is {\em $(0,1)$-affinoid} if there exists an affine $k$-variety $X=\Spec(A)$ with a closed subvariety $Z=\Spec(A/(f))$ and open subvariety $U=\Spec(A_f)$ such that $W$ is isomorphic to $X_\eta-U_\eta-Z_\eta$.
\end{defin}

\subsection{A closer look at $W$: Cartier case}
In general, $W$ can be covered by $(0,1)$-affinoid open subdomains $W_i$, so $\Coh(W)$ can be expressed in terms of $\Coh(W_i)$ which are just categories of finite modules over $k$-algebras of the form $A((f))$. This observation is contained implicitly in the proof of Lemma \ref{affcaselem}, so we just have to explicate what happens to $W$ throughout the steps of that proof. The case when $Z$ is a Cartier divisor is easier, so we consider it first.

\begin{lem}\label{Wlem}
Assume that $X$ is a $k$-variety, $Z\into X$ is a Cartier divisor and $U=X-Z$. Choose any covering of $X$ by open affine subschemes $X_i=\Spec(A_i)$ such that each $Z_i=Z\times_XX_i$ is given by a single function $f_i$. Set $U_i=X_i-Z_i$. Then we have the following.

(i) The $k$-analytic space $W=X_\eta-U_\eta-Z_\eta$ is admissibly covered by $(0,1)$-affinoid subdomains $W_i=(X_i)_\eta-(U_i)_\eta-(Z_i)_\eta$ and $\Coh(W_i)\toisom\Coh(A_i((f_i)))$.

(ii) If an intersection $X_{ij}=X_i\cap X_j$ is affine, say $X_{ij}=\Spec(A_{ij})$, then the intersection $W_{ij}=W_i\cap W_j$ is $(0,1)$-affinoid and $\Coh(W_{ij})\toisom\Coh(A_{ij}((f_i)))$. In particular, if $X$ is separated then all intersections $W_{ij}$ are $(0,1)$-affinoid.
\end{lem}
\begin{proof}
Since $(U_i)_\eta=U_\eta\cap (X_i)_\eta$ and $(Z_i)_\eta=Z_\eta\cap(X_i)_\eta$, we have that $W_i=W\cap(X_i)_\eta$. Since $(X_i)_\eta$ form an admissible covering of $X_\eta$ (as these is a finite covering by compact domains), the covering $W=\cup_iW_i$ is admissible. The second assertion of (i) follows from Lemma \ref{affcaselem}(ii). To verify (ii), set $U_{ij}=U_i\cap U_j=U\cap X_{ij}$ and similarly for $Z_{ij}$ and note that $W_{ij}=(X_{ij})_\eta-(U_{ij})_\eta-(Z_{ij})_\eta$. Since $Z_{ij}$ is given by $f_i$ in $X_{ij}$, we obtain (ii).
\end{proof}

The lemma gives a simple way to compute $\Coh(W)$ when $X$ is separated. In general, the intersections $W_{ij}$ may be only finite unions of $(0,1)$-affinoid domains. So, one may first use the lemma to compute $\Coh(W_{ij})$ and $\Coh(W_{ijk})$, and then compute $\Coh(W)$ as a second step.

\subsection{A closer look at $W$: the general case}
If $Z$ is arbitrary it is still possible to describe a $(0,1)$-affinoid covering $W=\cup_i W_i$. The formulas for intersection are simple but lengthy, so we ignore them for shortness.

\begin{lem}\label{Wgenlem}
Assume that $X$ is a $k$-variety, $Z\into X$ is a closed subvariety and $U=X-Z$. Choose any covering of $X$ by open affine subschemes $X_i=\Spec(A_i)$, for each $i$ choose elements $f_{i1}\.f_{in_i}\in A_i$ such that $Z\times_XX_i$ is of the form $V(f_{i1}\.f_{in_i})$. Then the $k$-analytic space $W=X_\eta-U_\eta-Z_\eta$ is admissibly covered by $(0,1)$-affinoid subdomains $W_{ij}$ such that $\calO_{X_\eta}(W_{ij})=B_{ij}$ and $\Coh(W_{ij})\toisom\Coh(B_{ij})$ for $B_{ij}=A_i((f_{ij}))\left\{\frac{f_{i1}}{f_{ij}}\.\frac{f_{in_i}}{f_{ij}}\right\}$.
\end{lem}
\begin{proof}
As in the proof of the previous lemma, we can easily localize $X$ reducing to the case when $X=\Spec(A)$ and $Z=\Spec(A/(f_1\. f_n))$. Since the simultaneous vanishing locus of $f_i$'s on $X_\eta$ is $Z_\eta$, the $f_i$'s do not vanish simultaneously on $W$. In particular, $W$ is the union of subdomains $$W_j=W\left\{\frac{f_1}{f_j}\.\frac{f_n}{f_j}\right\}=\{x\in W|\ |f_i(x)|\le|f_j(x)|, 1\le i\le n\}$$
The covering is admissible because the conditions cutting $W_j$'s are closed, so it remains to check that each $W_j$ is $(0,1)$-affinoid and $\calO_{X_\eta}(W_j)$ is isomorphic to $B_j=A((f_j))\left\{\frac{f_1}{f_j}\.\frac{f_n}{f_j}\right\}$.

Fix $j$ and consider the $A$-algebra $A'=A[\frac{f_1}{f_j}\.\frac{f_n}{f_j}]\subseteq A_{f_j}$ and the morphism $X'=\Spec(A')\to X$ (it is the $j$-th chart of the blow up $\Bl_Z(X)\to X$). Let $U'$ and $Z'$ be the preimages of $U$ and $Z$. Since $Z'=V(f_j)$ and $A'((f_j))\toisom B_j$, it suffices to check that $W'=X'_\eta-U'_\eta-Z'_\eta$ is mapped isomorphically onto $W_j$ by the morphism $X'_\eta\to X_\eta$. For this we set $W_{[r,s]}=W_j\{r\le|f_j|\le s\}$ and $W'_{[r,s]}=W'\{s\le|f_j|\le r\}$, so that $W_j=\cup_{0<r\le s<1} W_{[r,s]}$ (because $0<|f_j|<1$ on $X_\eta-U_\eta\supset W$) and similarly $W'=\cup_{0<r\le s<1} W'_{[r,s]}$, and note that by a simple computation both $W_{[r,s]}$ and $W'_{[r,s]}$ are rational domains isomorphic to $\calM(A\{s^{-1}f_j,rf_j^{-1},\frac{f_1}{f_j}\.\frac{f_n}{f_j}\})$.
\end{proof}

\subsection{Invariance of $W$ under modifications}\label{invarsec}
The interpretation of $W$ in \S\ref{propersec} suggests that it should only depend on $U$ and the birational class of $X$. Namely, $W$ should not change when we modify $X$ so that $U$ is preserved.\footnote{This was already proven by Thuillier, see \cite[Prop. 1.11]{Th}. Our proof is different.}

\begin{prop}\label{invarprop}
Assume that $X$ is a $k$-variety, $U\into X$ is a dense open subvariety, and $f\:X'\to X$ is a proper morphism, which is an isomorphism over $U$. Then the morphism $f_\eta\:X'_\eta\to X_\eta$ maps $W'$ isomorphically onto $W$, where, as usually, we set $Z=X-U$, $Z'=X'-U$, $W=X_\eta-U_\eta-Z_\eta$, and $W'=X'_\eta-U_\eta-Z'_\eta$.
\end{prop}
\begin{proof}
One way to prove the proposition is to compare $W$ and $W'$ explicitly using Lemma \ref{Wlem}. We prefer another way, which is less computational and less elementary. (For simplicity, we will even use the Nagata compactification theorem, though this could be easily avoided with an appropriate version of Chow's Lemma.) Since the claim is local on $X$, we can assume that $X$ is affine. Fix a compactification $X\into\oX$, and let $\oZ$ be the closure of $Z$ and $\oU=\oX-\oZ$. Gluing $X'$ with $\oU$ along $U$ we obtain a separated morphism $X'\cup\oU\to\oX$, which is an isomorphism over $\oU$, and applying the Nagata compactification theorem we can extend the latter to a proper morphism $\of\:\oX'\to\oX$, so that $X'\cup\oU$ is dense in $\oX'$ and hence $\of$ is an isomorphism over $\oU$. Let $\oZ'$ be the preimage of $\oZ$. Then $\oW'=\oX'_\eta-\oU_\eta-\oZ'_\eta$ and $\oW=\oX_\eta-\oU_\eta-\oZ_\eta$ are isomorphic to $\oU_\infty$ by \S\ref{propersec}, hence $\of_\eta$ maps $\oW'$ isomorphically onto $\oW$. It remains to note that $W'=X'_\eta\cap\oW'$ is the preimage of $W=X_\eta\cap\oW$ because $X'_\eta$ is the preimage of $X_\eta$ under $\of_\eta$.
\end{proof}

\begin{rem}
Proposition \ref{invarprop} implies that $W$ is preserved when we replace $X$ with its blow up $X'$ along $Z$ (recall that $X'\to X$ is an isomorphism over $U$). Note that the preimage $Z'$ of $Z$ is a Cartier divisor, i.e. it is locally given by a single equation $f=0$. In particular, Lemma \ref{Wlem} applies to $W'$. On the other hand, it is easy to see that our direct description of $W$ in Lemma \ref{Wgenlem} is nothing else than the description of $W'$ in terms of Lemma \ref{Wlem} applied to the blow up charts.
\end{rem}

\subsection{Other approaches}\label{othersec}
In this section we briefly discuss two other ``geometric realizations" of the punctured tubular neighborhood of $Z$. Both interpret $W$ as the generic fiber of the formal scheme $\gtX$, but this time we do not view it as a $k$-formal scheme. The construction goes as follows: first one blows up $X$ along $Z$ reducing the construction to the case when $Z$ is a Cartier divisor, then one covers $X$ by $X_i=\Spec(A_i)$ so that $Z\cap X_i=V(f_i)$, and glues $W^*$ from $W^*_i$'s, which are (appropriate) spectra of the rings $A((f_i))$.

\subsubsection{Analytic $k$-spaces}\label{kspacesec}
One approach is to view $A((f_i))$ as a $k((f_i))$-affinoid algebra and to take $W^\an_i=\calM(A((f_i)))$. This is the approach outlined by Drinfeld in his letter to Berkovich. The technical obstacle is that our choice of the ground field $k((f_i))$ is not unique, and it is not clear in which category one should glue $W^\an$. If $Z$ is given by a global function $f\in\calO_X(X)$ then we can work with $k((f))$-analytic spaces, but this trick cannot work in general.

Recall that Berkovich considers in \cite{berbook} and \cite{berihes} the category of analytic $k$-spaces, whose objects consist of an analytic $k$-field $K$ and a $K$-analytic space $X$, and morphisms $(Y,L)\to(X,K)$ consist of an isometry $K\into L$ and a morphism $Y\to X\wtimes_KL$. This category is also too narrow because any of its objects possesses a canonical field of definition. The latter feature is too restricting even in some simple questions of Berkovich geometry; for example, if $X$ is $K$-affinoid and $X\to\tilX$ is the reduction map then the preimage of a non-closed point of $\tilX$ can be viewed as an analytic $K$-space, but the field of definition is non-canonical. Concerning the above example, it was checked by the second author (unpublished) that one can extend the category of analytic $K$-spaces as follows: one starts with Banach algebras that are affinoid over an analytic $K$-subfield (this subfield is not fixed) and considers all bounded $K$-homomorphisms between them. Then one considers their Berkovich spectra $\calM$ and glues general spaces imitating \cite[\S1]{berihes}. As an output one obtains a category of generalized analytic $K$-spaces which is larger than the category of analytic $K$-spaces in two senses: there are new objects (those that do not admit a global field of definition), and there are new morphisms between classical analytic $K$-spaces (those that do not preserve the field of definition).

Once the category of generalized analytic $k$-spaces is constructed, it is almost immediate that the gluing procedure from the beginning of \S\ref{othersec} constructs $W^\an$ as such a space. On the topological level, one has that $W\toisom W^\an\times(0,1)$. Also, it is rather immediate from Lemma \ref{Wlem} that $\Coh(W)\toisom\Coh(W^\an)$.

\subsubsection{Adic spaces}\label{adicspaces} Another way to consider the general fiber of $\gtX$ is to work with adic spaces of R. Huber, see \cite{Hub}. In this case one glues $W^\ad=\gtX_\eta^\ad$ from affine adic spaces $\Spa(A((f_i)),A[[f_i]])$. Note also, that there is an adic analog of formal schemes, for example, $\gtX^\ad$ is glued from the affine adic spaces $\Spa(A[[f_i]],A[[f_i]])$, and $\gtX_\eta^\ad$ is an open subspace of $\gtX^\ad$ locally given by the non-vanishing of $f_i$ (so, it can be literally viewed as the generic fiber of $\gtX^\ad$). It seems that using adic spaces one can extend our results to the case when $X$ is an arbitrary noetherian scheme.

\section{Examples and discussion}\label{ExampleDiscuss}
\subsection{Examples}
\begin{exam}
Let $X$ be a separated $k$-variety and suppose we choose a finite open affine cover by varieties $X_{i}= \Spec(A_{i})$ where $A_{i}$ are $k$-algebras of finite type and the natural maps $A_{i} \to \mathcal{O}(X_{i} \times_{X} X_{j})$ are all localizations by some elements $g_{i,j} \in A_{i}$.   Suppose that $Z$ is defined in $X_{i}$ by one equation $f_{i} \in A_{i}$. Then the categories $\Coh(X)$, $\Coh(\widehat{X}_{Z})$, $\Coh(U)$ and $\Coh(W)$ can all be described in terms of modules.  Namely, $\Coh(X)$ is a category of $M_{i} \in \Coh(A_{i})$ and gluing isomorphisms $(M_{i})_{g_{ij}} \cong (M_{j})_{g_{ji}}$ in $\Coh((A_{i})_{g_{ij}})$ which satisfy the cocycle condition.   $\Coh(U)$ is a category of $N_{i} \in \Coh((A_{i})_{f_{i}})$ and gluing isomorphisms $(N_{i})_{g_{ij}} \cong (N_{j})_{g_{ji}}$ in $\Coh((A_{i})_{(f_ig_{ij})})$ which satisfy the cocycle condition.    $\Coh(\widehat{X}_{Z})$ is a category of finitely generated modules $F_{i} \in \Coh(\widehat{A_{i}})$ and gluing isomorphisms $(F_{i})_{\{g_{ij}\}} \cong (F_{j})_{\{g_{ji}\}}$ in $\Coh((\widehat{A_{i}})_{\{g_{ij}\}})$ which satisfy the cocycle condition.  Finally, $\Coh(W)$ is isomorphic to the category of finitely generated modules $G_{i}\in\Coh((\widehat{A_{i}})_{f_{i}})$, together with gluing isomorphisms of the induced modules in $\Coh((\widehat{(A_{i})_{g_{ij}}})_{f_{i}})$ which satisfy the cocycle condition.  Notice that in the above description, we implicitly used the isomorphisms $(A_i)_{g_{ij}} \cong (A_j)_{g_{ji}}, (A_{i})_{(f_ig_{ij})} \cong (A_{j})_{(f_jg_{ji})},  (\widehat{A_{i}})_{\{g_{ij}\}} \cong (\widehat{A_{j}})_{\{g_{ji}\}},$ and $(\widehat{(A_{i})_{g_{ij}}})_{f_{i}} \cong (\widehat{(A_{j})_{g_{ji}}})_{f_{j}}$. On the affine level, the four functors appearing in the fiber product (\ref{FibProd}) can be translated via this description into functors $\Coh(A_{i}) \to \Coh(\widehat{A_{i}})$, $\Coh(\widehat{A_{i}}) \to \Coh((\widehat{A_{i}})_{f_{i}})$, $\Coh(A_{i}) \to \Coh((A_{i})_{f_{i}})$, and $\Coh((A_{i})_{f_{i}}) \to \Coh((\widehat{A_{i}})_{f_{i}})$. Each of these functors is simply a tensor product with the appropriate ring.  Therefore, locally, the descent setup we described is isomorphic to descent found either in \cite{A} or in \cite{BL2} and discussed in Section \ref{relatedsec}.
\end{exam}

\begin{exam}\label{ProjSpace}
The key to most applications of Theorem \ref{MainThm} will be an understanding of the space $W$.  For instance, to study vector bundles, it will be important to understand the group of invertible matrices with values in the functions on $W$.  We give here a consistency check that shows the Theorem \ref{MainThm} makes sense in a simple example.  We use the groups $\Pic$ which assign to a space the group of line bundles modulo equivalence.  Let $k$ be any field and let $X=\mathbb{P}^{r}_{k}$, where $r \geq 1$ and let $\mathbb{P}^{r-1}_{k} =Z \into X$ be the inclusion of a hyperplane.    The formal scheme completion of $X$ along $Z$ is denoted by $\gtX=\widehat{X}_{Z}$.  Notice that $U=X-Z=\mathbb{A}^{n}_{k}$, and so $\Pic(U)=\{1\}$.
Consider the sequence of groups given by the restriction maps
\[\Pic(X) \to \Pic(\widehat{X}_{Z}) \to \Pic(Z).
\]
When $r>1$ the composition is an isomorphism so the first map is an injection and the second map is a surjection even though the middle term is infinite dimensional.  When $r=1$ the second and third terms are trivial.  Notice that Theorem \ref{MainThm} implies that the kernel of the first map is precisely the double quotient
\begin{equation} \mathcal{O}(U)^{\times} \backslash \mathcal{O}(W)^{\times}/\mathcal{O}(\widehat{X}_{Z})^{\times}.
\end{equation}
An element in $\mathcal{O}(W)^{\times}$ can be seen as defining a line bundle on $X$ which is glued from the trivial line bundle on $\widehat{X}_{Z}$ and the trivial line bundle on $U$.  Therefore, the equivalence classes of line bundles on $X$ are the quotient of  $\mathcal{O}(W)^{\times}$ by the automorphisms of the trivial line bundle on  $\widehat{X}_{Z}$ and on $U$.  Notice that we have $\mathcal{O}(U)^{\times}=k^{\times}$ and it is not hard to see that $\mathcal{O}_{\gtX} \cong \prod^{\infty}_{m=0} \mathcal{O}_{\mathbb{P}_{k}^{r-1}}(-m)$ as a sheaf of $\mathcal{O}_{\mathbb{P}^{r-1}_{k}}$-modules.  Therefore $\mathcal{O}(\widehat{X}_{Z})=\text{H}^{0}(\mathbb{P}^{r-1}_{k}, \prod^{\infty}_{m=0} \mathcal{O}_{\mathbb{P}^{r-1}_k}(-m))$ and so $\mathcal{O}(\widehat{X}_{Z})^{\times}=\mathcal{O}(Z)^{\times}= k^{\times}$ when $r>1$ and $\mathcal{O}(\widehat{X}_{Z})^{\times}=k[[t]]^{\times}$ when $r=1$.  Therefore, we would like to check directly that this kernel is the correct group, in other words,
\begin{equation}
\begin{split}
 \mathcal{O}(W)^{\times}/k[[t]]^{\times}
                             & \cong \mathbb{Z} \ \  when \ \ r=1 \\
 \mathcal{O}(W)^{\times}/k^{\times}                             & \cong  \{1\} \ \ when \ \ r >1.
\end{split}
\end{equation}
We should have $ \mathcal{O}(W)^{\times}=k((t))^{\times}$ for $r=1$ and $ \mathcal{O}(W)^{\times}=k^{\times}$ for $r>1$.
In this example, for $r>1$, we have an exact sequence
\[ \{1\} \to \Pic(X) \to \Pic(\widehat{X}_{Z}) \to \Pic(W)
.\]
Even though $\Pic(\widehat{X}_{Z})$ is huge for $r>1$, only a subgroup isomorphic to $\mathbb{Z}$ is in the kernel of the restriction map to $\Pic(W)$.

Now let us, indeed, compute $\calO(W)$ and verify the above predictions. Let $X=\mathbb{P}^{r}_{k}=\Proj(x_0\. x_r)$ and fix a hyperplane $Z=V(x_0)=\mathbb{P}^{r-1}_{k}=\Proj(x_1\. x_r)$. Set $\gtX=\hatX_Z$ and $W=\gtX_\infty=\gtX_\eta-(\gtX_s)_\eta$ (we use notation from Remark \ref{Wrem}). For any $1\le i\le r$ let $Z_i=\mathbb{A}^{r-1}_k=\Spec(k[\frac{x_1}{x_i}\.\frac{x_r}{x_i}])$ be the non-vanishing locus of $x_i$. The open affine covering $Z=\cup_{i=1}^rZ_i$ induces an open affine covering of $\gtX$ by formal schemes $\gtX_i=\Spf(k[[\frac{x_0}{x_i}]]\{\frac{x_1}{x_i}\.\frac{x_r}{x_i}\})$, where we substitute $1$ instead of $\frac{x_i}{x_i}$. Let $\emptyset\neq I\subseteq\{1\.r\}$ then $\gtX_I:=\cap_{i\in I}\gtX_i=\Spf(k[[\frac{x_0}{x_{i_0}}]]\{\frac{x_j}{x_k}\})$, where $i_0$ is any element of $I$, $j$ and $k$ run over all pairs with $1\le j\le r$, $k\in I$, and the fractions $\frac{x_j}{x_k}$ satisfy all natural relations (e.g. $\frac{x_j}{x_k}\frac{x_k}{x_l}=\frac{x_j}{x_l}$).

It follows from Lemma \ref{affcaselem} that $W_I:=(\gtX_I)_\infty$ satisfies $\calO(W_I)=k((\frac{x_0}{x_{i_0}}))\{\frac{x_j}{x_k}\}$, where the convergent power series over $k((\frac{x_0}{x_{i_0}}))$ form a Tate algebra (discussed in \S\ref{kAffinoid}) over the field $k((\frac{x_0}{x_{i_0}}))$. In particular, we can compute $\calO(W)$ as $\cap_{i=1}^r\calO(W_i)$ (more precisely, we use the information about the gluing homomorphisms $\calO(W_i)\to\calO(W_{i,j})$) and after a simple computation one obtains that $\calO(W)=k[\frac{x_1}{x_0}\.\frac{x_r}{x_0}]$ whenever $r>1$. In particular, $\calO^\times(W)=k^\times$, as predicted earlier. For $r=1$, no computation is needed as we have that $\calO(W)=\calO(W_1)=k((\frac{x_0}{x_1}))$.
\end{exam}

\subsection{Discussion}
Theorem \ref{MainThm} can be applied to the study of stacks of vector bundles on algebraic surfaces $X$ over the site of $k$-schemes of finite type.  The simplest case is when $Z$ is taken to be a rational curve inside $X$.  The local behavior of $\widehat{X}_{Z}$ depends very much on the self-intersection number of the curve, if the curve is contractible (as in the case of negative self-intersection) then a reasonable notion of a formal neighborhood can be constructed as an actual scheme instead of a formal scheme.  In the case of negative self-intersection the contraction is a (singular) surface.  D. Harbater suggested to the first author that one can define a scheme version of the formal neighborhood as the fiber product (over the contraction) of the surface with the formal completion of the contraction at the image point of the curve.  This is precisely the type of descent which was considered by Artin in \cite{A}.  The case of negative self-intersection is opposite to Example \ref{ProjSpace} in the sense that with negative self-intersection one has $\Pic(\widehat{X}_{Z})=\Pic(Z)$ but the space of functions $\mathcal{O}(\widehat{X}_{Z})$ is huge, whereas in Example \ref{ProjSpace} it is the other way around: $\mathcal{O}(\widehat{X}_{Z}) = \mathcal{O}(Z)$ while $\Pic(\widehat{X}_{Z})$ is huge.  A local study of stacks of vector bundles on this type of scheme (equivalently on $\widehat{X}_{Z}$) was done in \cite{locsur}.  One might try to produce a general construction by taking the relative Spec over $X$ of the {\it coherator} of the sheaf of functions on  $\widehat{X}_{Z}$ thought of as a sheaf of $\mathcal{O}_{X}$-modules.  The coherator converts $\mathcal{O}_{X}$-modules to quasi-coherent $\mathcal{O}_{X}$-modules.   It was introduced by Grothendieck and a review of its properties can be found in Appendix B of \cite{TT}.  However when applied to the case of positive self-intersection in situation such as in Example \ref{ProjSpace} for $r=2$, this method would fail because the coherator is very small and represents functions defined on some Zariski open set containing $Z$.  Fixing this situation by giving a general construction that works with any subvariety and over any field was one of the motivations of the present paper.  Even in the case that $k=\mathbb{C}$ and $Z$ and $X$ are smooth, it is not clear how to replicate the results in this article using the methods of the classical theory of complex manifolds.  This is because the direct limit of the categories $\Coh(V)$ for $V$ running over open sets in the classical topology which contain $Z$ need not coincide with $\Coh(\widehat{X}_{Z})$.  A descent statement using this kind of direct limit will appear in \cite{BB}.  The advantage of involving $\Coh(\widehat{X}_{Z})$ is that it is described (Chapter II, Proposition 9.6 of \cite{Ha}) as being an inverse limit of the categories of coherent sheaves over finite length infinitesimal neighborhoods of $Z$.

\end{document}